\documentclass[12pt]{amsart}  

\usepackage{amsmath,amsthm,amssymb,amscd,mathdots,enumerate,shuffle,hyperref}
\usepackage{tikz}
\allowdisplaybreaks

\newtheorem{theorem}{Theorem}[section]

\newtheorem{definition}[theorem]{Definition}
\newtheorem{lemma}[theorem]{Lemma}
\newtheorem{corollary}[theorem]{Corollary}

\newtheorem{ex}[theorem]{Example}
\newtheorem{remark}[theorem]{Remark}

\numberwithin{equation}{section}

\newcommand{\Q}{\mathbb{Q}}
\newcommand{\Z}{\mathbb{Z}}
\newcommand{\C}{\mathbb{C}}

\newcommand{\End}{\operatorname{End}}

\newcommand{\SL}{\operatorname{SL}}
\newcommand{\Li}{\operatorname{Li}}

\newcommand{\zav}{\hat{\zeta}}
\newcommand{\ze}{\zeta^{\bf e}}
\newcommand{\zo}{\zeta^{\bf o}}
\newcommand{\zoe}{\zeta^{\bf oe}}
\newcommand{\zee}{\zeta^{\bf ee}}

\newcommand{\qrs}{q^f_{r,s}}

\newcommand{\abcd}{\begin{pmatrix} a & b \\ c & d \end{pmatrix}}
\title{Modular forms and $q$-analogues of\\ modified Double Zeta Values}

\author{Henrik Bachmann}
\address{Graduate School of Mathematics,  Nagoya University, Nagoya, Japan.}
\email{henrik.bachmann@math.nagoya-u.ac.jp}


\begin{document}
\date{\today}
\maketitle

\markright{\uppercase{Modular forms and} $q$-\uppercase{analogues of modified Double Zeta Values}}

\begin{abstract} 
We present explicit formulas for Hecke eigenforms as linear combinations of q-analogues of modified double zeta values. As an application, we obtain period polynomial relations and sum formulas for these modified double zeta values. These relations have similar shapes as the period polynomial relations of Gangl, Kaneko and Zagier and the usual sum formulas for classical double zeta values. 
\end{abstract}

\section{Introduction}

In \cite{GKZ} Gangl, Kaneko and Zagier gave an explicit connection between cusp forms for the full modular group of weight $r+s$  and  $\Q$-linear relations among the double zeta values
\[ \zeta(r,s) = \sum_{0 < m < n} \frac{1}{m^{r} n^{s}} \,,\quad (r\geq 1, s\geq 2) \,.\]
For example, one of the consequences of their work (see Remark \ref{rem:gkzrel}) is that the first non-trivial cusp form $\Delta(q)=q\prod_{n>0}(1-q^n)^{24}$ in weight $12$ gives rise to the relation

\begin{equation}\label{gkzrel}
14 \zeta(3,9) +42 \zeta(4,8)+75 \zeta(5,7)+95 \zeta(6,6)+84 \zeta(7,5) + 42 \zeta(8,4) = \frac{6248}{691} \zeta(12) \,.
\end{equation}
The connection of this relation to the cusp form $\Delta$ is given by the fact that the coefficients on the left-hand side are obtained by the even period polynomial of $\Delta$.

In this note we will show a similar result for the following modified version of the double zeta values
\begin{equation}\label{eq:zav}
\zav(r,s) =  \sum_{0 < m < n} \frac{1}{ (m+n)^{r} n^{s}} \,,\quad (r\geq 1, s\geq 2) \,,
\end{equation}
and give an even more direct connection between cusp forms and linear relation among them. 
The values \eqref{eq:zav} are special cases of Apostol-Vu double zeta values or Witten zeta functions for $\mathfrak{so}(5)$ (see \cite{Mat,O}).
As an analogue of the relation \eqref{gkzrel} we obtain
\begin{equation}\label{eq:zavrel}
14 \zav(3,9) +42 \zav(4,8)+75 \zav(5,7)+95 \zav(6,6)+84 \zav(7,5) + 42 \zav(8,4) = \frac{1639}{176896}\zeta(12) \,,
\end{equation}
which is not a trivial consequence of \eqref{gkzrel}, since it is expected that $\zav(r,s)$ is in general not a linear combination of $\zeta(r,s)$ and $\zeta(r+s)$. 
The connection of relation \eqref{eq:zavrel} and the cusp form $\Delta$ will be made explicit by writing $\Delta$ as a linear combination of $q$-analogues of the modified double zeta values  $\zav(r,s)$ and the Riemann zeta value $\zeta(k)$. These are $q$-series $\zav_q(r,s), \zeta_q(k) \in \Q[[q]]$ which degenerate to $\zav(r,s)$ and $\zeta(k)$ respectively when $q\rightarrow 1$ (see Lemma \ref{lem:qana}). In general we will write any Hecke eigenform as a linear combination of these $q$-analogues plus a "lower-weight" $q$-series, which vanishes as $q\rightarrow 1$. We denote by $M_k$ and $S_k$ the spaces of modular forms and cusp forms of weight $k$ for the full modular group. The first result of this work is the following.

\begin{theorem}\label{thm:main1} Let $f \in S_k$ be a cuspidal Hecke eigenform with restricted even period polynomial $P^{ev,0}_f$ (see \eqref{eq:defp0} for the definition). Define the coefficients $\qrs \in \C$ by
\begin{equation}\label{eq:qrs}
P^{ev,0}_f(X+Y,X) = \sum_{\substack{r+s=k\\ r,s\geq 1}} \binom{k-2}{r-1} \qrs X^{r-1} Y^{s-1}\,. 
\end{equation}
Then $f$ can be written as
\begin{align*}
 \frac{L_f^*(1)}{2(k-2)!}f(q) = \sum_{\substack{r+s=k\\ r,s\geq 2}} \qrs \,\zav_q(r,s) - \lambda_f \zeta_q(k) - R_f(q) \,,
\end{align*}
where $L^\ast_f$ is the completed L-function of $f$ (see \eqref{eq:lser}), $R_f(q) \in \C[[q]]$ is an explicitly given "lower weight" $q$-series (see Lemma \ref{lem:t2}) and where
\begin{equation}\label{eq:lamf}
\lambda_f = \frac{k-1}{2} \left( \sum_{\substack{r+s=k\\r,s\geq 3 \text{ odd}}} \frac{(-1)^{\frac{s-1}{2}}}{r\, 2^{r-1}} \binom{k-2}{s-1} L_f^*(s) -L_f^*(1) \right)\,.
\end{equation}
\end{theorem}
We will see that for a cusp form $f \in S_k$ the $R_f(q)$ and $f(q)$ vanish as $q \rightarrow 1$ (after multiplying with $(1-q)^k$). As a corollary of our result we therefore obtain the following analog of the result of Gangl, Kaneko and Zagier for the values $\zav(r,s)$. 
\begin{corollary}\label{cor:main1} For a cuspidal Hecke eigenform $f\in S_k$ the following relation holds
\[\sum_{\substack{r+s = k\\r,s \geq 2}} \qrs \zav(r,s) = \lambda_f \, \zeta(k) \,, \]   
where the coefficients $ \qrs$ and $\lambda_f$ are given by \eqref{eq:qrs} and \eqref{eq:lamf} respectively.
\end{corollary}
By the work of Kohnen and Zagier (\cite{KZ}) is is known that there exists a basis of Hecke eigenforms $\{f_i\}$ for $S_k$, such that $P^{ev,0}_{f_i}(X,Y)\in \Q[X,Y]$. Therefore Corollary \ref{cor:main1} gives $\dim S_k$-many $\Q$-linear relations among the modified double zeta values.
As the second result of this work we will write the $q$-analogue $\zeta_q(k)$, which is just the Eisenstein series of weight $k$ without constant term, as a sum over all $\zav_q(r,s) $ with $r+s=k$ and another explicitly given "lower-weight" $q$-series $E_k(q)$.

\begin{theorem} \label{thm:main2} For all even $k\geq 4$ we have
\[ \zeta_q(k) = 2^{k-1} \sum_{\substack{r+s=k\\r\geq 1,s\geq 2}} \zav_q(r,s)  - E_k(q)\,,\]
where the $q$-series $E_k(q) \in \Q[[q]]$ is given by \eqref{eq:ek}.
\end{theorem}
Again by considering $q \rightarrow 1$ the $E_k(q)$ vanishes and we get, for the even weight case (the odd weight case will be proven separately),  the following sum formula.
\begin{theorem} \label{cor:main2} For all $k\geq 3$ we have
\[   \zeta(k) = 2^{k-1}\sum_{\substack{r+s=k\\r\geq 1,s\geq 2}} \zav(r,s)\,.\]
\end{theorem}

The contents of this paper are as follows. In Section \ref{sec:qana} we start by giving the definition of the $q$-analogues of the modified double zeta values $\zav(r,s)$. For the proof of Theorem \ref{thm:main1} and \ref{thm:main2} we need the theory of Hecke operators for period polynomials of modular forms, which we will introduce in Section \ref{sec:perpol}. Finally, we write any modular form as a linear combination of $q$-analogues in Section \ref{sec:mfasqana} and give the proofs of the main results.

\subsection*{Acknowledgment}
The author would like to thank Ulf K\"uhn and Nils Matthes for fruitful comments and corrections on an early draft of this work. 

\section{$q$-analogues of modified double zeta values}\label{sec:qana}
In this section we will introduce $q$-analogues of the modified double zeta value $\zav(r,s)$. For classical double (or multiple) zeta values there are various works on different models of $q$-analogues in the literature. An easy way to obtain a $q$-analogue of a zeta value is to replace the appearing natural numbers $n$ in the definition by their $q$-analogues $[n]_q = \frac{1-q^n}{1-q}$, which satisfy $\lim_{q\rightarrow 1} [n]_q = n$. In general a sum of the form 
\[\sum_{0 < m < n} \frac{Q_r(q^{m+n}) }{ [m+n]_q^{r} } \frac{ Q_s(q^{n})}{  [n]_q^{s}} \,, \]
where  $Q_r(t), Q_s \in t \Q[t]$ are polynomials satisfying $Q_r(1)=Q_s(1)=1$, gives a $q$-analogue of $\zav(r,s)$. In the context of modular forms it is convenient to remove the global factor $(1-q)^{r+s}$ in the definition of these $q$-analogues and to use the  polynomials $Q_k \in  t \Q[t]$ defined for $k\geq 1$ by the identity
\begin{equation}\label{eq:qk}
 \frac{Q_k(t)}{(1-t)^k} = \frac{1}{(k-1)!} \sum_{d>0} d^{k-1} t^d\,. 
\end{equation}
We have $Q_1(t) = t$ and for $k\geq 2$ the $Q_k(t)$ are polynomials of degree $k-1$ satisfying $Q_k(1)=1$. These are up to a factor the so called Eulerian polynomials (c.f. \cite[Remark 2.6]{BK}). 
\begin{definition} \label{def:qzeta}For $k,r, s\geq 1$ we define  the $q$-analogues of $\zeta(k)$ and $\zav(r,s)$  by 
\begin{align*} 
\zeta_q(k) &= \sum_{n>0} \frac{Q_k(q^n)}{(1-q^n)^k}\,,
\\
\zav_q(r,s) &= \sum_{0 < n < m} \frac{Q_r(q^{n+m}) }{ (1-q^{n+m})^{r} } \frac{ Q_s(q^{m})}{ (1-q^m)^{s}} \,,
\end{align*}
where the polynomials $Q_j(t)$ for $j\geq 1$ are defined by \eqref{eq:qk}.
\end{definition}

\begin{lemma} \label{lem:qana}
\begin{enumerate}[i)]
\item For $k\geq 2$ and  $r\geq 1, s\geq 2$ we have
\[\lim_{q\rightarrow 1} (1-q)^{k} \zeta_q(k) = \zeta(k)\,,\qquad \lim_{q\rightarrow 1} (1-q)^{r+s} \zav_q(r,s) = \zav(r,s)\,. \]
In particular $\lim_{q\rightarrow 1} (1-q)^{k} \zeta_q(k') = \lim_{q\rightarrow 1} (1-q)^{k} \zav_q(r,s)  = 0$ if $k', r+s < k$.
\item If $f(q) = \sum_{n\geq 0} a_n q^n \in M_k$ is a modular form of weight $k$, then 
 \[ \lim_{q\rightarrow 1} (1-q)^{k} f(q) = (-2\pi i)^k a_0 \,.\]
In particular $\lim_{q\rightarrow 1}(1-q)^{k} f(q) = 0$ if $f \in S_k$ is a cusp form.
\end{enumerate}
\end{lemma}
\begin{proof}
This follows from Proposition 6.4 and Corollary 6.5 in \cite{BK}, where the notation $[k]= \zeta_q(k)$ is used. The result for $\zav_q(r,s)$ follows with a similar argument as given there for the $q$-series $[r,s]$.
\end{proof}

\section{Period polynomials and Hecke operators}\label{sec:perpol}
We recall the definition and results on period polynomials as they are presented in \cite{Z1}, \cite{Z2} and \cite{Z3}.
Denote for even $k\geq 4$ by $V_k \subset \C[X,Y]$ the space of homogeneous polynomials in two indeterminates of degree $k-2$. The group $\SL_2(\Z)$ acts on the space $V_k$ by
\begin{equation}\label{eq:defactions}
(P|\gamma)(X,Y) = P(aX+bY,cX+dY)\,\qquad \left(P \in V_k\,, \gamma = \abcd \in  \SL_2(\Z)\right).
\end{equation}
Further denote by $S$ and $U$ the following elements in  $\SL_2(\Z)$ 
\[ S = \begin{pmatrix} 0 & -1 \\ 1 & 0 \end{pmatrix}\,,\qquad U = \begin{pmatrix} 1 & -1 \\ 1 & 0 \end{pmatrix}  \,.\]
For a modular form $f(\tau) = \sum_{n\geq 0} a_n q^n \in M_k$, where as usual $\tau$ is an element in the complex upper-half plane and $q=\exp(2\pi i \tau)$, define the even (extended) period polynomial of $f$ by
\[P^{ev}_f(X,Y) = \sum_{\substack{r+s=k\\r,s\geq 1 \text{ odd}}} (-1)^{\frac{s-1}{2}} \binom{k-2}{s-1} L_f^*(s) X^{r-1} Y^{s-1} \in V_k\,. \] 
Here for $\Re(s) \gg 0$ the $L_f^*(s)$ denotes the L-series $L_f(s)=\sum_{n\geq 1} a_n n^{-s}$ of $f$ multiplied by its gamma-factor
\begin{equation}\label{eq:lser}
L_f^*(s) = \int_0^{\infty} \left( f(i y) - a_0 \right) y^{s-1} dy = (2\pi)^{-s} \Gamma(s) L(f,s)\,. 
\end{equation}
The function $L_f^*(s)$ has a meromorphic continuation to all $s$, with simple poles at $s=0$ and $s=k$, and satisfies the functional equation $L_f^*(s) = (-1)^{\frac{k}{2}}L_f^*(k-s)$. Using the modular transformation of $f$, one can check that $P^{ev}_f$ vanishes under the action of $1+S$ and $1+U+U^2$ and therefore it is an element in the space
\[W_k = \left\{ P \in V_k \mid P|(1+S) = P|(1+U+U^2) = 0  \right\}   \,.\]
We decompose $W_k = W^{ev}_k \oplus W^{od}_k $ into the even and odd polynomials and therefore have $P^{ev}_f \in W^{ev}_k$.
As a generalization of the classical Eichler-Shimura isomorphism, which deals with the case of $f$ being a cusp form, Zagier proved the following.
\begin{theorem}(\cite{Z2})\label{thm:eichlershimura}
The map $f \mapsto P^{ev}_f$ is an isomorphism from $M_k$ to $W^{ev}_k$.
\end{theorem}
One of the most important structures on the space $M_k$ is the action of the Hecke algebra. For $n \in \Z_{\geq 1}$ denote by $T_n \in \End(M_k)$ the $n$-th Hecke operator. Due do Theorem \ref{thm:eichlershimura} a natural question is, if there is an operator on $W_k^{ev}$, which corresponds to the operator $T_n$ on $M_k$.
One such operator was first given in \cite{Z1} and to define it we first write $\operatorname{M}_n = \left\{ {\tiny \abcd} \mid a,b,c,d \in \Z, ad-bc=n \right\}$ and extend the action \eqref{eq:defactions} linearly to an action of the group ring $\Q[\operatorname{M}_n]$ on $V_k$. For $n\in \Z_{\geq 1}$ we then define the element
\[ \tilde{T}_n  = \sum_{\substack{ad-bc=n \\ a > c > 0\\d > -b > 0}} \left( \abcd + \begin{pmatrix} a & -b \\ -c & d \end{pmatrix} \right) + \sum_{\substack{ad=n\\ -\frac{d}{2} < b \leq \frac{d}{2} }} \begin{pmatrix} a & b \\ 0 & d \end{pmatrix} + \sum_{\substack{ad=n\\ -\frac{a}{2} < c \leq \frac{a}{2} \\ c \neq 0}} \begin{pmatrix} a & 0\\ c & d \end{pmatrix} \in \Q[\operatorname{M}_n] \,.  \]

\begin{theorem}\label{thm:hoponperiod}
The action of $\tilde{T}_n$ on $W_k^{\text{ev}}$ corresponds to the action of $T_n$ on $W_k$, i.e. we have for all $f \in M_k$
\begin{equation}\label{eq:tnthm}
 P^{ev}_{T_n f}(X,Y) = P^{ev}_f |{\tilde{T}_n}(X,Y)\,. 
\end{equation}
\end{theorem}
\begin{proof}
This is Theorem 2 in \cite{Z1} or Theorem 3 in \cite{CZ}.
\end{proof}

\section{Modular forms as $q$-analogues of double zeta values} \label{sec:mfasqana}
To make notations shorter we define the following pairing of a polynomial $P(X,Y) \in \C[X,Y]$ and an element $T = \sum_{\gamma} \alpha_\gamma \gamma \in \Q[M_n]$
\begin{align*}
\langle P , T \rangle := \sum_{\gamma \,=\, {\tiny \abcd}} \alpha_\gamma P(b,d)\,.
\end{align*}
With this we obtain the following consequence of Theorem \ref{thm:hoponperiod}, which gives an explicit formula for the Fourier coefficients of Hecke eigenforms. 
\begin{lemma}\label{lem:anformula} Let $f = \sum_{n\geq 0} a_n q^n \in M_k$ be a Hecke eigenform, i.e. $T_nf = a_n f$, then we have for $n\geq 1$
\[a_n = -\frac{1}{L_f^\ast(1)} \langle P^{ev}_f , \tilde{T}_n \rangle \,.\]
\end{lemma}
\begin{proof} It is is well-known that zeros of $L_f^\ast(s)$, for a cuspidal Hecke eigenform $f$, can only occur  inside the critical strip $\frac{k-1}{2}<\Re(s)<\frac{k+1}{2}$, and in particular $L_f^\ast(1) \neq 0$. Also $L_{G_k}^\ast(1) \neq 0$ for the normalized Eisenstein series $G_k$.
Setting $(X,Y)=(0,1)$ in \eqref{eq:tnthm}, the left-hand side becomes  $P^{ev}_{T_n f}(0,1) = a_n P^{ev}_{f}(0,1) = a_n (-1)^{\frac{k-2}{2}}L_f^\ast(k-1) =-a_n L_f^\ast(1)$ and the right-hand side is by definition of the pairing given by  $P^{ev}_f |{\tilde{T}_n}(1,0) = \langle P^{ev}_f , \tilde{T}_n \rangle $, from which the statement follows. 
\end{proof}

\begin{remark}
A similar formula as in Lemma \ref{lem:anformula} for the Fourier coefficients of cusp forms was already given by Manin in \cite[Section 1.3]{Man}. 
\end{remark}
\begin{corollary}\label{cor:divisorsum}For even $k\geq 4$ and $n\geq 1$ we have
\[\sigma_{k-1}(n) = \sum_{d|n} d^{k-1} = \langle  Y^{k-2}-X^{k-2} , \tilde{T}_n \rangle\,. \]
\end{corollary}
\begin{proof}
This follows directly from Lemma \ref{lem:anformula}, since the normalized Eisenstein series $G_k(\tau)= -\frac{B_k}{2k} + \sum_{n>0} \sigma_{k-1}(n) q^n$ is a Hecke eigenform with $P^{ev}_{G_k} = L^*_{G_k}(1) (X^{k-2}-Y^{k-2})$ (see the first proposition in Section 2 of \cite{Z2}). 
\end{proof}

To proof Theorem \ref{thm:main1} we will calculate $\langle P^{ev}_f , \tilde{T}_n \rangle$ explicitly. First we define the \emph{even restricted period polynomial} $P^{ev,0}_f$ of a modular form $f \in M_k$  by
\begin{align}\begin{split} \label{eq:defp0}
P^{ev,0}_f(X,Y) &= P^{ev}_f(X,Y) - L_f^*(1)(X^{k-2}-Y^{k-2}) \\
&=\sum_{\substack{r+s=k\\r,s\geq 3 \text{ odd}}} (-1)^{\frac{s-1}{2}} \binom{k-2}{s-1} L_f^*(s) X^{r-1} Y^{s-1} \,. 
\end{split}
\end{align}

\begin{lemma}\label{lem:1}
For a cuspidal Hecke eigenform $f \in S_k$ we have
\[ f(q) =  (k-1)! \zeta_q(k) - \frac{1}{L_f^\ast(1)} \sum_{n>0} \langle P^{ev,0}_f , \tilde{T}_n \rangle q^n\,.\] 
\end{lemma}
\begin{proof}
This follows by Corollary \ref{cor:divisorsum} together with \eqref{eq:defp0} and the fact that the coefficients of $\zeta_q(k)$ are given by the divisor-sum $\sigma_{k-1}(n)$, since
 \[\zeta_q(k) = \sum_{n>0} \frac{Q_k(q^n)}{(1-q^n)^k} =  \frac{1}{(k-1)!}\sum_{n>0} \sum_{d>0} d^{k-1} q^{dn} = \frac{1}{(k-1)!} \sum_{n>0} \sigma_{k-1}(n) q^n \,.\]
\end{proof}

It remains to evaluate $\langle P^{ev,0}_f , \tilde{T}_n \rangle$. For this we write $\tilde{T}_n = \tilde{T}^{(1)}_n + \tilde{T}^{(2)}_n + \tilde{T}^{(3)}_n$ with
\begin{align*}
\tilde{T}^{(1)}_n &= \sum_{\substack{ad-bc=n \\ a > c > 0\\d > -b > 0}} \left( \abcd + \begin{pmatrix} a & -b \\ -c & d \end{pmatrix} \right)\,, \\
\tilde{T}^{(2)}_n &= \sum_{\substack{ad=n\\ -\frac{d}{2} < b \leq \frac{d}{2} }} \begin{pmatrix} a & b \\ 0 & d \end{pmatrix} \,,\qquad \tilde{T}^{(3)}_n= \sum_{\substack{ad=n\\ -\frac{a}{2} < c \leq \frac{a}{2} \\ c \neq 0}} \begin{pmatrix} a & 0\\ c & d \end{pmatrix}\,.
\end{align*}
In the following we will calculate $\sum_{n>0} \langle P^{ev,0}_f , \tilde{T}^{(j)}_n \rangle q^n$ individually before combining them in the end for the proof of Theorem \ref{thm:main1}.

\begin{lemma} \label{lem:t1}For a cusp form $f \in S_k$ we have 
\begin{align*}
 \sum_{n>0} \langle P^{ev,0}_f , \tilde{T}^{(1)}_n \rangle q^n = -2(k-2)!\sum_{\substack{r+s=k\\ r,s\geq 2}} q^f_{r,s} \,\zav_q(r,s) \,,
\end{align*}
where the coefficients $ q^f_{r,s} $ are given by \eqref{eq:qrs}.
\end{lemma}
\begin{proof}
By direct calculation and the fact that $P^{ev,0}$ is an even polynomial, we obtain 
\begin{align*}
\sum_{n>0} \langle P^{ev,0}_f , \tilde{T}^{(1)}_n \rangle q^n  &= \sum_{\substack{a > c > 0\\d > -b > 0}} \left(P_f^{ev,0}(b,d) + P_f^{ev,0}(-b,d) \right) q^{ad-bc} \\
&= 2 \sum_{\substack{a > c > 0\\d > b > 0}} P_f^{ev,0}(b,d) q^{ad+bc} =  2 \sum_{\substack{a > c > 0\\d,b > 0}} P_f^{ev,0}(b,b+d) q^{a(d+b)+bc}\,.
\end{align*}
Using $P_f^{ev,0}(b,b+d) = - P_f^{ev,0}(b+d,b)$ and the definition of $\qrs$ as coefficients of $P_f^{ev,0}(X+Y,X)$,  we can write 
\begin{align}\label{eq:lem451}
\begin{split}
\sum_{n>0} \langle P^{ev,0}_f , \tilde{T}^{(1)}_n \rangle q^n  &= -2 \sum_{\substack{r+s=k\\ r,s\geq 1}} \binom{k-2}{r-1} \qrs \sum_{\substack{a > c > 0\\d,b > 0}} b^{r-1} d^{s-1} q^{a(d+b)+bc}\\
&=-2(k-2)!  \sum_{\substack{r+s=k\\ r,s\geq 1}} \qrs \sum_{\substack{a > c > 0\\d,b > 0}} \frac{b^{r-1}}{(r-1)!} \frac{d^{s-1}}{(s-1)!}q^{(a+c)b+ad} \,.
\end{split}
\end{align}
By the definition of $\zav(r,s)$ (see Definition \ref{def:qzeta}) we have
\begin{equation}\label{eq:lem452}
\zav_q(r,s) = \sum_{a > c > 0} \frac{Q_r(q^{a+c}) }{ (1-q^{a+c})^{r} } \frac{ Q_s(q^{a})}{ (1-q^a)^{s}} = \sum_{\substack{a > c > 0\\d,b > 0}} \frac{b^{r-1}}{(r-1)!} \frac{d^{s-1}}{(s-1)!}q^{(a+c)b+ad}\,. 
\end{equation}
Since $P_f^{ev,0}(X,0)=P_f^{ev,0}(0,Y)=0$ and $P_f^{ev,0}|(1+U+U^2)=0$ it follows that $ q^f_{r,1}= q^f_{1,s} =0$. Combining this together with \eqref{eq:lem451} and \eqref{eq:lem452} we obtain the desired result. 
\end{proof}

To evaluate $\sum_{n>0} \langle P^{ev,0}_f , \tilde{T}^{(2)}_n \rangle q^n$ we will introduce some further notation. For $k\geq 1$ we define the \emph{even and odd $q$-analogues of the single zeta value} by
\[\ze_q(k) = \sum_{\substack{a,d>0\\d \text{ even}}} \frac{d^{k-1}}{(k-1)!} q^{ad}\,,\qquad \zo_q(k) = \sum_{\substack{a,d>0\\d \text{ odd}}} \frac{d^{k-1}}{(k-1)!} q^{ad}\,.\] 

\begin{lemma} \label{lem:t2} For a cusp form $f \in S_k$ with even restricted period polynomial 
\begin{align*}
P^{ev,0}_f(X,Y)=\sum_{\substack{r+s=k\\r,s\geq 3 \text{ odd}}} c_{r,s} X^{r-1} Y^{s-1}
\end{align*}
we have 
\begin{align*}
 \sum_{n>0} \langle P^{ev,0}_f , \tilde{T}^{(2)}_n \rangle q^n = \Bigg( \sum_{\substack{r+s=k\\r,s\geq 3 \text{ odd}}} \frac{ c_{r,s}}{r 2^{r-1}}\Bigg)  (k-1)!\zeta_q(k) + 2(k-2)! R_f(q)\,,
\end{align*}
where the $q$-series $R_f(q)$ is given by
\begin{align*}
 R_f(q) &= \sum_{\substack{r+s=k\\r,s\geq 3 \text{ odd}}} c_{r,s}  \left( \sum_{j=1}^{r-1} \binom{r}{j}\frac{B_j \cdot (k-j-1)!}{r 2^{r-j} (k-2)!}  \,\zeta_q(k-j) - \frac{1}{2^{r}}\, \ze_q(k-1) \right)\\
&+ \sum_{\substack{r+s=k\\r,s\geq 3 \text{ odd}}} c_{r,s}  \sum_{\substack{0 \leq j \leq r-1\\1 \leq l \leq r-j}} \binom{r}{j} \binom{r-j}{l} \frac{(-1)^l B_j \cdot (k-j-l-1)!}{r 2^{r-j} (k-2)!} \,\zo_q(k-j-l)\,.
\end{align*}

\end{lemma}
\begin{proof}
Again by using the fact that $P_f^{ev,0}$ is even and $P_f^{ev,0}(0,Y)=0$  we obtain
\begin{align}\label{eq:lemt21}
 \sum_{n>0} \langle P^{ev,0}_f , \tilde{T}^{(2)}_n \rangle q^n 
&=\sum_{\substack{r+s=k\\r,s\geq 3 \text{ odd}}} c_{r,s} \Bigg( 2 \sum_{\substack{a,d > 0\\ 0 < b \leq \frac{d}{2} }} b^{r-1} d^{s-1} q^{ad} -  \frac{1}{2^{r-1}} \sum_{\substack{a,d>0\\d \text{ even}}} d^{k-2} q^{ad} \Bigg) \,.
\end{align}
Now by using the well-known formula
\[ \sum_{b=1}^{N} b^{r-1} =  \frac{1}{r} \sum_{j=0}^{r-1} \binom{r}{j} B_j N^{r-j} \,, \]
where $B_j$ denotes the Bernoulli numbers of the second kind, i.e. $B_1 = \frac{1}{2}$, we get by a straightforward calculation for $r+s=k$
\begin{align*}
\sum_{\substack{a,d > 0\\ 0 < b \leq \frac{d}{2} }} b^{r-1} d^{s-1} q^{ad} &= \frac{(k-1)!}{r 2^r} \zeta_q(k) + \sum_{j=1}^{r-1} \binom{r}{j}\frac{B_j \cdot(k-j-1)!}{r 2^{r-j}} \zeta_q(k-j) \\
&+ \sum_{\substack{0 \leq j \leq r-1\\1 \leq l \leq r-j}} \binom{r}{j} \binom{r-j}{l} \frac{(-1)^l  B_j \cdot (k-j-l-1)! }{r 2^{r-j}}\zo_q(k-j-l)\,.
\end{align*}
Combining this with \eqref{eq:lemt21}, yields the result stated. 
\end{proof}
\begin{proof}[Proof of Theorem \ref{thm:main1}]
The statement of Theorem $\ref{thm:main1}$ follows by combining  Lemma \ref{lem:1} ,\ref{lem:t1} ,\ref{lem:t2} and the fact that $ \sum_{n>0} \langle P^{ev,0}_f , \tilde{T}^{(3)}_n \rangle q^n = 0$.
\end{proof}

\begin{proof}[Proof of Corollary \ref{cor:main1}]
In \cite[Proposition 7.2]{BK} it was shown that for a $q$-series $f(q)=\sum_{n>0} a_n q^n$ with $a_n = O(n^{K-1})$ and $K < k$ one has $\lim_{q\rightarrow 1} (1-q)^k f(q)=0$. The coefficients of $\zeta_q(K),\ze_q(K)$ and $\zo_q(K)$ are all in $O(n^{K-1})$ and since in the definition of $R_f$ just the cases $K < k$ (where $k$ is the weight of $f$) appear, we get $\lim_{q\rightarrow 1} (1-q)^k R_f(q)=0$. Corollary \ref{cor:main1} therefore follows from Lemma \ref{lem:qana} and Theorem \ref{thm:main1}.
\end{proof}

\begin{ex}
We give one example for Theorem \ref{thm:main1}. The period polynomial of $f = (45 L^*_{\Delta}(9))^{-1} \Delta$ is given by
\[ P^{ev}_f(X,Y)=\frac{36}{691}(X^{10}-Y^{10}) - X^2 Y^2 (X^2-Y^2)^3\,.\]
In this case the $q$-series $R_f(q)$ can be written as
\begin{align*}
R_f(q) &= \frac{1}{5} \zeta_q(4) +\frac{40}{21} \zeta_q(6) +21 \zeta_q(8)- \frac{51}{128} \zo_q(4)  - \frac{15}{4}  \zo_q(6)- \frac{315}{8} \zo_q(8)\,. 
\end{align*}
By Theorem \ref{thm:main1} we therefore obtain the following expression for $\Delta$
\begin{align*}
\frac{\Delta(q)}{221120} &= \frac{1639}{176896} \zeta_q(12) - \frac{1}{11520} R_f(q)\\
&-\left(14 \zav_q(3, 9)+42 \zav_q(4, 8)+75 \zav_q(5, 7) + 95 \zav_q(6, 6) +
 84 \zav_q(7, 5) + 42 \zav_q(8, 4) \right)\,,
\end{align*}
from which the relation \eqref{eq:zavrel} in the introduction follows after multiplying both sides by $(1-q)^{12}$ and taking the limit $q\rightarrow 1$.
\end{ex}
\begin{proof}[Proof of Theorem \ref{thm:main2}] By Corollary \ref{cor:divisorsum} we have
\begin{align*}
 (k-1)! \zeta_q(k) = \sum_{n>0} \langle Y^{k-2}- X^{k-2} , \tilde{T}_n \rangle q^n \,.
\end{align*} 
With similar calculations as in Lemma \ref{lem:t1} and \ref{lem:t2} one can give explicit formulas for $\sum_{n>0} \langle Y^{k-2}- X^{k-2} , \tilde{T}^{(1)}_n \rangle q^n$ and $\sum_{n>0} \langle Y^{k-2}- X^{k-2} , \tilde{T}^{(2)}_n \rangle q^n$. Together with 
\begin{align*}
\sum_{n>0} \langle Y^{k-2}- X^{k-2} , \tilde{T}^{(3)}_n \rangle q^n &= (k-3)! q \frac{d}{dq} \zeta_q(k-2)-(k-2)! \zeta_q(k-1)
\end{align*}
one then can check that Theorem \ref{thm:main2} holds with the $q$-series $E_k(q)$ given by
\begin{align}\label{eq:ek}\begin{split}
E_k(q) =& \,2^{k-2} \zeta_q(k-1)-\frac{2^{k-2} }{(k-2)} q \frac{d}{dq} \zeta_q(k-2) + \sum_{j=2}^{k-2} \frac{2^j B_j }{j!} \zeta_q(k-j)\\
&+ \sum_{\substack{0 \leq j \leq k-2\\1 \leq l \leq k-j-1\\(l,j)\neq(1,0)}} \frac{(-1)^l 2^j B_j }{j!\, l!} \zo_q(k-j-l)\,.
\end{split}
\end{align}
\end{proof}

\begin{proof}[Proof of Theorem \ref{cor:main2}]
For even weight $k$ Theorem \ref{cor:main2} follows from Theorem  \ref{thm:main2} with the same arguments as given for Corollary \ref{cor:main1}. To prove the odd weight case, we first observe that the modified double zeta value can be written as
\begin{align}\begin{split}\label{eq:mdavasli}
\zav(r,s)&=\sum_{0 < m < n} \frac{1}{ (m+n)^{r} n^{s}} = \sum_{0 < m < n < 2n} \frac{1}{ m^{r} n^{s}} = \left(\sum_{0 < m < 2n} - \sum_{0 < m < n} - \sum_{0 < m = n} \right)\frac{1}{ m^{r} n^{s}} \\
&=2^{s-1}  (\Li_{r,s}(-1)  + \zeta(r,s)) - \zeta(r,s) - \zeta(r+s)\,,
\end{split}
\end{align}
where $\Li_{r,s}(z)=\sum_{0 < m < n} \frac{z^n}{m^{r} n^{s}}$ denotes the double polylogarithm. When $k=r+s$ is odd, it is known, due to the parity result for double polylogarithms  (see \cite[(75)]{BBB}), that $\Li_{r,s}(z)$ can be  written explicitly in terms of single polylogarithms. From this one can deduce together with $\Li_k(1) + \Li_k(-1)=\frac{1}{2^{k-1}} \Li_k(1) = \frac{1}{2^{k-1}} \zeta(k)$, that for odd $k$
\[\sum_{\substack{r+s=k\\r\geq 1,s\geq 2}} 2^{s-1} \Li_{r,s}(-1) = \frac{1}{2^{k-1}} \zeta(k) + \frac{k-3}{2} \zeta(k)\,. \]
Now using the following sum formulas for double zeta values (see \cite{OZ}) 
\begin{align*}
\sum_{\substack{r+s=k\\r\geq 1,s\geq 2}} \zeta(r,s) = \zeta(k) \,,\qquad \sum_{\substack{r+s=k\\r\geq 1,s\geq 2}} 2^{s-1}\zeta(r,s) = \frac{(k+1)}{2}\zeta(k)\,,
\end{align*}
we obtain together with \eqref{eq:mdavasli} 
\begin{align*}
\sum_{\substack{r+s=k\\r\geq 1,s\geq 2}} \zav(r,s) &= \frac{1}{2^{k-1}} \zeta(k) + \frac{k-3}{2} \zeta(k)+ \frac{(k+1)}{2}\zeta(k) - (k-1)\zeta(k) = \frac{1}{2^{k-1}} \zeta(k)\,.
\end{align*}
\end{proof}
We end this note by giving examples for Theorem \ref{thm:main2} and some general remarks.

\begin{ex} For $k=4,6$ Theorem \ref{thm:main2} gives the following expressions for $\zeta_q(k)$.
\begin{align*}
\zeta_q(4) &= 8 \left( \zav_q(1,3)+ \zav_q(2,2) \right)- \frac{1}{3} \zeta_q(2)-4 \zeta_q(3)+\frac{1}{2} \zo_q(2) +2 q\frac{d}{dq}\zeta_q(2)\,,\\
\zeta_q(6) &=32 \left( \zav_q(1,5)+\zav_q(2,4)+\zav_q(3,3)+\zav_q(4,2)\right)+\frac{1}{45} \zeta_q(2) - \frac{1}{3}  \zeta_q(4) -16 \zeta_q(5)\\
&-\frac{1}{24}  \zo_q(2)+\frac{1}{2} \zo_q(4)+4 q\frac{d}{dq}\zeta_q(4)\,.
\end{align*}
\end{ex}

\begin{remark}
Numerically also Theorem \ref{thm:main2} holds for any $k\geq 3$ and Theorem \ref{cor:main2} should be a Corollary of this general version. But since we are using the period polynomials of the Eisenstein series, our proof of Theorem \ref{thm:main2} is just valid for even $k \geq 4$.
\end{remark}
\begin{remark}\label{rem:gkzrel}
The result for the classical case of double zeta values, given in \cite[Theorem 3]{GKZ},  focuses on relations among $\zeta(r,s)$, where $r$ and $s$ are both odd. Their result is that for a cusp forms $f\in S_k$ the following relation holds
\begin{equation}\label{eq:gkzthm}
\sum_{\substack{r+s = k\\r,s \geq 3 \text{:odd}}} \qrs \zeta(r,s) = \beta_f \, \zeta(k) \,, 
\end{equation} 
where the coefficients $\qrs$ are given by \eqref{eq:qrs} and the coefficient $ \beta_f$ (first explicitly written down by Ma and Tasaka in \cite[Corollary 2.3]{MT}) is given by
\[ \beta_f = -\frac{1}{2} \left(\frac{k-1}{2} L_f^*(1) +\sum_{\substack{r+s = k\\r,s \geq 3 \text{:odd}}} \qrs \right)\,.\]
For integers $a,b \in \Z_{\geq 1} $ the coefficients $\qrs$ satisfy $q^f_{2a,2b}=q^f_{2b,2a}$. Together with the well-known fact that $\zeta(2a)\zeta(2b)\in \zeta(2(a+b))\Q$ and the harmonic product formula
\[\zeta(2a)\zeta(2b)=\zeta(2a,2b)+\zeta(2b,2a)+\zeta(2(a+b))\,,\]
one obtains that also $\sum_{\substack{r+s = k\\r,s \geq 2}} \qrs \zeta(r,s)$ is a multiple of $\zeta(k)$. This gives the relation \eqref{gkzrel} in the introduction as a consequence of the famous relation
\[28 \zeta(3,9)+150 \zeta(5,7) + 168 \zeta(7,5) = \frac{5197}{691}\zeta(12)\,, \]
which follows from \eqref{eq:gkzthm} by taking for $f$ a certain multiple of $\Delta$. For the modified double zeta values the harmonic product formula does not hold and therefore it is not clear if one can reduce our result to the case where $r$ and $s$ are both odd.
\end{remark}

\begin{remark}
In \cite{KT} the authors introduced (using a different order) double zeta values of level $2$ given for $r\geq 1, s\geq 2$ by 
\begin{align*}
\zoe(r,s) = \sum_{\substack{0 < m < n\\m \text{ odd},\, n \text{ even}}} \frac{1}{m^{r} n^{s}} \,,\quad\zee(r,s) = \sum_{\substack{0 < m < n\\m \text{ even},\, n \text{ even}}} \frac{1}{m^{r} n^{s}} \,.
\end{align*}
These are related to the modified and the usual double zeta values by (see \eqref{eq:mdavasli})
\begin{align*}
\zav(r,s)=2^s  (\zoe(r,s)  + \zee(r,s)) - \zeta(r,s) - \zeta(r+s)\,.
\end{align*}
Combining Corollary \ref{cor:main1} and the period polynomial relations for classical double zeta values (Remark \ref{rem:gkzrel}), one could therefore also explicitly write down  period polynomial relations and sum formulas for the values $2^s  (\zoe(r,s)  + \zee(r,s))$.
\end{remark}

\end{document}